\documentclass[10pt]{article}
\usepackage{amsmath}
\usepackage{amssymb}
\usepackage{amsthm}
\usepackage{amsfonts}

\begin{document}

\def\D{\displaystyle}
\newtheorem{theorem}{Theorem}[section]
\newtheorem{definition}[theorem]{Definition}
\newtheorem{lemma}[theorem]{Lemma}
\newtheorem{algorithm}[theorem]{Algorithm}
\newtheorem{example}[theorem]{Example}
\newtheorem{proposition}[theorem]{Proposition}
\renewcommand{\theequation}{\thesection. \arabic{equation}}

\begin{center}
\large {\bf Multivariate Difference-Differential Dimension Polynomials and New Invariants of Difference-Differential Field Extensions }
\normalsize
\smallskip

Alexander  Levin

The Catholic University of America

 Washington, D. C.  20064

levin@cua.edu

http://faculty.cua.edu/levin

\end{center}

\begin{abstract}
In this paper we introduce a method of characteristic sets with respect to several term orderings for difference-differential polynomials.  Using this technique, we obtain a method of computation of multivariate dimension polynomials of finitely generated difference-differential field extensions. Furthermore, we find new invariants of such extensions and show how the computation of multivariate difference-differential polynomials is applied to the equivalence problem for systems of algebraic difference-differential equations.
\end{abstract}



\section{Introduction}
The role of Hilbert polynomials in commutative and homological algebra as well as in algebraic geometry and combinatorics is well known. A similar role
in differential algebra is played by differential dimension polynomials, which describe in exact terms the freedom degree of a dynamic system,
as well as the number of arbitrary constants in the general solution of a system of partial algebraic differential equations. The notion of a differential dimension polynomial was introduced by E. Kolchin \cite{K1} who proved the following fundamental result.

\begin{theorem} Let $K$ be a differential field of zero characteristic with basic derivations $\delta_{1},\dots, \delta_{m}$. Let $\Theta$ denote the free commutative semigroup generated by $\delta_{1},\dots, \delta_{m}$, and for any $r\in {\bf N}$, let $\Theta(r) = \{\theta = \delta_{1}^{k_{1}}\dots \delta_{m}^{k_{m}}\in\Theta \,|\,\sum_{i=1}^{m}k_{i}\leq r\}$. Furthermore, let $L = K\langle \eta_{1},\dots,\eta_{n}\rangle$ be a differential field extension of $K$ generated by a finite set $\eta = \{\eta_{1}, \dots , \eta_{n}\}$.
Then there exists a polynomial $\omega_{\eta|K}(t)\in {\bf Q}[t]$ such that $\omega_{\eta|K}(r) = trdeg_{K}K(\{\theta \eta_{j} | \theta \in \Theta(r), \,1\leq j\leq n\})$ for all sufficiently large $r\in {\bf Z}$. The degree of this polynomial does not exceed $m$ and the numbers  $d =  \deg \omega_{\eta|K}$,\, $a_{m}$ and $a_{d}$ do not depend on the choice of the system of differential generators $\eta$ of the extension $L/K$.  Moreover, $a_{m}$ is equal to the differential transcendence degree of $L$ over $K$, that is, to the maximal number of elements $\xi_{1},\dots,\xi_{k}\in L$ such that the set $\{\theta \xi_{i} | \theta \in \Theta, 1\leq i\leq k\}$ is algebraically independent over $K$.
\end{theorem}

The polynomial $\omega_{\eta|K}(t)$ is called the {\em differential dimension polynomial} of the extension $L/K$ associated with the set of differential generators $\eta$.

If $P$ is a prime differential ideal of a finitely generated differential algebra $R = K\{\zeta,\dots, \zeta_{n}\}$ over a differential field $K$, then the quotient field  of $R/P$ is a differential field extension of $K$ generated by the images of $\zeta_{i}$ ($1\leq i\leq n$) in $R/P$. The corresponding differential dimension polynomial, therefore, characterizes the ideal $P$; it is denoted by $\omega_{P}(t)$.  Assigning such polynomials to prime differential ideals has led to a number of new results on the Krull-type dimension of differential algebras and dimension of differential varieties (see, for example, \cite{Johnson1}, \cite{Johnson2} and \cite{Johnson3}). Furthermore, as it was shown by A. Mikhalev and E. Pankratev \cite{MP1}, one can naturally assign a differential dimension polynomial to a system of algebraic differential equations and this polynomial expresses the A. Einstein's strength of the system (see \cite{E}). Methods of computation of (univariate) differential dimension polynomials and the strength of systems of differential equations via the Ritt-Kolchin technique of characteristic sets can be found, for example, in \cite{MP2} and \cite[Chapters 5, 9]{KLMP}. Note also, that there are quite many works on computation of dimension polynomials of differential, difference and difference-differential modules with the use of various generalizations of the Gr\"obner basis method (see, for example, \cite[Chapters V - XI]{KLMP}, \cite{Levin1}, \cite{Levin2}, \cite{Levin3}, \cite[Chapter 3]{Levin4}, and \cite{ZW}).

In this paper we develop a method of characteristic sets with respect to several orderings for algebras of difference-differential polynomials over a difference-differential fields whose basic set of derivations is partied into several disjoint subsets. We apply this method to prove the existence, outline a method of computation, and determine invariants of a multivariate dimension polynomial associated with a finite system of generators of a difference-differential field extension (and a partition of the basic sets of derivations).  We also show that most of these invariants are not carried by univariate dimension polynomials and show how the consideration of the new invariants can be applied to the isomorphism problem for difference-differential field extensions and equivalence problem for systems of algebraic difference-differential equations.

\section{Preliminaries}

Throughout the paper, ${\bf N}, {\bf Z}$, ${\bf Q}$, and ${\bf R}$ denote the sets
of all non-negative integers, integers, rational numbers, and real numbers, respectively.
${\bf Q}[t]$ will denote the ring of polynomials in one variable $t$
with rational coefficients.  %and for any $n\in {\bf N}$, $o(t^{n})$ will denote a polynomial in ${\bf Q}[t]$ of degree less than $n$.

By a {\em difference-differential ring} we mean a commutative ring $R$ together with finite sets $\Delta = \{\delta_{1},\dots, \delta_{m}\}$ and $\sigma = \{\alpha_{1},\dots, \alpha_{n}\}$ of derivations and automorphisms of $R$, respectively, such that any two mappings of the set $\Delta\bigcup\sigma$ commute.  The set $\Delta\bigcup\sigma$ is called the {\em basic set\/} of the difference-differential ring $R$, which is also called a $\Delta$-$\sigma$-ring. If $R$ is a field, it is called a {\em difference-differential field} or a $\Delta$-$\sigma$-field. Furthermore, in what follows, we denote the set $\{\alpha_{1},\dots, \alpha_{n}, \alpha^{-1}_{1},\dots, \alpha^{-1}_{n}\}$ by $\sigma^{\ast}$.

If $R$ is a difference-differential ring with a basic set $\Delta\bigcup\sigma$ described above, then
$\Lambda$ will denote the free commutative semigroup of all power products of the form $\lambda = \delta_{1}^{k_{1}}\dots \delta_{m}^{k_{m}}\alpha_{1}^{l_{1}}\dots \alpha_{n}^{l_{n}}$ where $k_{i}\in {\bf N},\, l_{j}\in {\bf Z}$ ($1\leq i\leq m,\, 1\leq j\leq n$). For any such an element $\lambda$, we set $\lambda_{\Delta} = \delta_{1}^{k_{1}}\dots \delta_{m}^{k_{m}}$, $\lambda_{\sigma} = \alpha_{1}^{l_{1}}\dots \alpha_{n}^{l_{n}}$, and denote by $\Lambda_{\Delta}$ and $\Lambda_{\sigma}$ the commutative semigroup of power products $\delta_{1}^{k_{1}}\dots \delta_{m}^{k_{m}}$ and the commutative group of elements of the form $\alpha_{1}^{l_{1}}\dots \alpha_{n}^{l_{n}}$, respectively.
The {\em order} of $\lambda$ is defined as $ord\,\lambda = \sum_{i=1}^{m}k_{i} + \sum_{j=1}^{n}|l_{j}|$, and for every $r\in {\bf N}$, we set $\Lambda(r) = \{\lambda\in \Lambda\,|\, ord\,\lambda\leq r\}$ ($r\in {\bf N}$).

\smallskip

A subring (ideal) $R_{0}$ of a $\Delta$-$\sigma$-ring $R$ is said to be a difference-differential  (or $\Delta$-$\sigma$-) subring
of $R$ (respectively, difference-differential (or $\Delta$-$\sigma$-) ideal of $R$) if $R_{0}$ is closed with respect to the action of any operator of $\Delta\bigcup\sigma^{\ast}$. If a prime ideal $P$ of $R$ is closed with respect to the action of $\Delta\bigcup\sigma^{\ast}$, it is called a {\em prime} difference-differential (or $\Delta$-$\sigma$-) {\em ideal \/} of $R$.

If $R$ is a $\Delta$-$\sigma$-field and $R_{0}$ a subfield of $R$ which is also a $\Delta$-$\sigma$-subring of $R$, then  $R_{0}$ is said to be a
$\Delta$-$\sigma$-subfield of $R$; $R$, in turn, is called a difference-differential (or $\Delta$-$\sigma$-) field extension or a $\Delta$-$\sigma$-overfield of $R_{0}$. In this case we also say that we have a $\Delta$-$\sigma$-field extension $R/R_{0}$.

If $R$ is a $\Delta$-$\sigma$-ring and $\Sigma \subseteq R$, then the intersection of all $\Delta$-ideals of $R$ containing the set $\Sigma$ is, obviously, the smallest $\Delta$-$\sigma$-ideal of $R$ containing $\Sigma$. This ideal is denoted by $[\Sigma]$. (It is clear that $[\Sigma]$ is generated, as an ideal, by the set $\{\lambda \xi | \xi \in \Sigma,\, \lambda\in \Lambda\}$). If the set $\Sigma$ is finite, $\Sigma = \{\xi_{1},\dots, \xi_{q}\}$, we say that the $\Delta$-ideal $I = [\Sigma]$ is finitely generated (we write this as $I = [\xi_{1},\dots, \xi_{q}]$) and call $\xi_{1},\dots, \xi_{q}$ difference-differential (or $\Delta$-$\sigma$-) generators of $I$.

If $K_0$ is a $\Delta$-$\sigma$-subfield of the $\Delta$-$\sigma$-field $K$ and $\Sigma \subseteq K$, then the intersection of all $\Delta$-$\sigma$-subfields
of $K$ containing $K_0$ and $\Sigma$ is the unique $\Delta$-$\sigma$-subfield of $K$ containing $K_0$ and $\Sigma$ and contained in every $\Delta$-$\sigma$-subfield of $K$ containing $K_0$ and $\Sigma$. It is denoted by $K_{0}\langle \Sigma \rangle$. If $K = K_{0}\langle \Sigma \rangle$ and the set $\Sigma$ is finite, $\Sigma = \{\eta_{1},\dots,\eta_{s}\}$, then $K$ is said to be a finitely generated $\Delta$-$\sigma$-extension of $K_{0}$ with the set of $\Delta$-$\sigma$-generators $\{\eta_{1},\dots,\eta_{s}\}$. In this case we write
$K = K_{0}\langle \eta_{1},\dots,\eta_{s}\rangle$.  It is easy to see that the field $K_{0}\langle \eta_{1},\dots,\eta_{s}\rangle$ coincides with the
field $K_0(\{\lambda \eta_{i} | \lambda \in \Lambda, 1\leq i\leq s\}$).

Let $R$ and $S$ be two difference-differential rings with the same basic set $\Delta\bigcup\sigma$, so that elements of the sets $\Delta$ and $\sigma$ act on each of the rings as mutually commuting derivations and automorphisms, respectively. A ring homomorphism $\phi: R \longrightarrow S$ is called a {\em difference-differential\/} (or $\Delta$-$\sigma$-) {\em homomorphism\/} if $\phi(\tau a) = \tau \phi(a)$ for any $\tau \in \Delta\bigcup\sigma$, $a\in R$.

\smallskip

If $K$ is a difference-differential ($\Delta$-$\sigma$-) field and $Y =\{y_{1},\dots, y_{s}\}$ is a finite set of symbols, then one can consider the countable set of symbols $\Lambda Y = \{\lambda y_{j}|\lambda \in \Lambda, 1\leq j\leq s\}$ and the polynomial ring $R = K[ \{\lambda y_{j}|\lambda \in \Lambda, 1\leq j\leq s\}]$ in the set of indeterminates $\Lambda Y$ over the field $K$. This polynomial ring is naturally viewed as a $\Delta$-$\sigma$-ring where $\tau(\lambda y_{j}) = (\tau\lambda)y_{j}$ for any $\tau \in \Delta\bigcup \sigma$, $\lambda \in \Lambda$, $1\leq j\leq s$, and the elements of $\Delta\bigcup\sigma$ act on the coefficients of the polynomials of $R$ as they act in the field $K$. The ring $R$ is called a {\em ring of difference-differential\/} (or $\Delta$-$\sigma$-) {\em polynomials\/} in the set of differential ($\Delta$-$\sigma$-)indeterminates $y_{1},\dots, y_{s}$ over $K$. This ring is denoted by $K\{y_{1},\dots, y_{s}\}$ and its elements are called difference-differential (or $\Delta$-$\sigma$-) polynomials.

\smallskip

Let $L = K\langle \eta_{1},\dots,\eta_{s}\rangle$ be a difference-differential field extension of $K$ generated by a finite set $\eta =\{\eta_{1},\dots,\eta_{s}\}$. As a field, $L = K(\{\lambda\eta_{j} | \lambda\in \Lambda, 1\leq j\leq s\})$.

The following is a unified version of E. Kolchin's theorem on differential dimension polynomial and the author's theorem on the dimension polynomial of a difference field extension (see \cite{Levin1} or ~\cite[Theorem 4.2.5]{Levin4}\,).

\begin{theorem}  With the above notation, there exists a polynomial $\phi_{\eta|K}(t)\in {\bf Q}[t]$ such that

{\em (i)}\, $\phi_{\eta|K}(r) = trdeg_{K}K(\{\lambda\eta_{j} | \lambda\in \Lambda(r), 1\leq j\leq s\})$ for all sufficiently large $r\in {\bf Z}$;

{\em (ii)}\, $\deg \phi_{\eta|K} \leq m+n$ and $\phi_{\eta|K}(t)$ can be written as \, $\phi_{\eta|K}(t) = \D\sum_{i=0}^{m+n}a_{i}{t+i\choose i}$ where $a_{0},\dots, a_{m+n}\in {\bf Z}$ and $2^{n}|a_{m+n}$\,\,.

{\em (iii)}\, $d =  \deg \phi_{\eta|K}$,\, $a_{m+n}$ and $a_{d}$ do not depend on the set of difference-differential generators $\eta$ of $L/K$ ($a_{d}\neq a_{m+n}$ if and only if $d < m+n$). Moreover, $\D\frac{a_{m+n}}{2^{n}}$ is equal to the difference-differential transcendence degree of $L$ over $K$ (denoted by $\Delta$-$\sigma$-$trdeg_{K}L$), that is, to the maximal number of elements $\xi_{1},\dots,\xi_{k}\in L$ such that the family
$\{\lambda \xi_{i} | \lambda \in \Lambda, 1\leq i\leq k\}$ is algebraically independent over $K$.
\end{theorem}
The polynomial whose existence is established by this theorem is called a {\em univariate difference-differential} (or $\Delta$-$\sigma$-) {\em dimension polynomial} of the extension $L/K$ associated with the system of difference-differential generators $\eta$.

\section{Partition of the basic set of derivations and the formulation of the main theorem}

Let $K$ be a difference-differential field of zero characteristic with basic
sets $\Delta = \{\delta_{1},\dots, \delta_{m}\}$ and $\sigma = \{\alpha_{1},\dots, \alpha_{n}\}$
of derivations and automorphisms, respectively. Suppose that the set of derivations is
represented as the union of $p$ disjoint subsets ($p\geq 1$):
\begin{equation}\Delta = \Delta_{1}\bigcup \dots \bigcup \Delta_{p}\end{equation}
$$\text{where}\,\,\,  \Delta_{1} = \{\delta_{1},\dots, \delta_{m_{1}}\},\, \Delta_{2} = \{\delta_{m_{1}+1},\dots, \delta_{m_{1}+m_{2}}\}, \,\dots,$$
$$\Delta_{p} = \{\delta_{m_{1}+\dots + m_{p-1}+1},\dots, \delta_{m}\}\, \,\, \, (m_{1}+\dots + m_{p} = m).$$
If \,\,$\lambda  = \delta_{1}^{k_{1}}\dots \delta_{m}^{k_{m}}\alpha_{1}^{l_{1}}\dots \alpha_{n}^{l_{n}}\in\Lambda$
($k_{i}\in {\bf N}, \,\, l_{j}\in {\bf Z}$), then the order of $\lambda$ with respect to $\Delta_{i}$  ($1\leq i\leq p$)
is defined as $\D\sum_{\nu = m_{1}+\dots + m_{i-1}+1}^{m_{1}+\dots + m_{i}}k_{\nu}$; it is denoted by $ord_{i}\lambda$.
(If $i=1$, the last sum is replaced by $k_{1}+\dots + k_{m_{1}}$.)
The number $ord_{\sigma}\lambda = \D\sum_{j = 1}^{n}|l_{j}|$ is called the order of $\lambda$ with respect to $\sigma$. Furthermore, for any
 $r_{1},\dots, r_{p+1}\in {\bf N}$, we set $$\Lambda(r_{1},\dots, r_{p+1}) = \{\lambda\in\Lambda\,|\,ord_{i}\lambda \leq r_{i} \,
(i=1,\dots, p)\,\,\, \text{and}\,\,\, ord_{\sigma}\lambda \leq r_{p+1}\}.$$
In what follows, for any permutation $(j_{1},\dots, j_{p+1})$ of the set $\{1,\dots, p+1\}$, $<_{j_{1},\dots, j_{p+1}}$ will denote the lexicographic order on ${\bf N}^{p+1}$ such that

\noindent$(r_{1},\dots, r_{p+1})<_{j_{1},\dots, j_{p+1}} (s_{1},\dots, s_{p+1})$ if and only if either $r_{j_{1}} < s_{j_{1}}$ or there exists $k\in {\bf N}$, $1\leq k\leq p$, such that $r_{j_{\nu}} = s_{j_{\nu}}$ for $\nu = 1,\dots, k$ and $r_{j_{k+1}} < s_{j_{k+1}}$.

Furthermore, if $\Sigma \subseteq {\bf N}^{p+1}$, then $\Sigma'$ denotes the set

\noindent$\{e\in \Sigma | e$ is a maximal element of $\Sigma$ with
respect to one of the $(p+1)!$ lexicographic orders
$<_{j_{1},\dots, j_{p+1}}\}$.

For example, if $\Sigma = \{(3, 0, 2), (2, 1, 1), (0, 1, 4), (1, 0, 3), (1, 1, 6), (3, 1, 0), \\(1, 2, 0)\} \subseteq {\bf N}^{3}$, then
$\Sigma' = \{(3, 0, 2), (3, 1, 0), (1, 1, 6), (1, 2, 0)\}$.

\begin{theorem} Let $L = K\langle \eta_{1},\dots,\eta_{s}\rangle$ be a
$\Delta$-$\sigma$-field extension generated by a set $\eta =
\{\eta_{1}, \dots , \eta_{s}\}$. Then there exists a polynomial
$\Phi_{\eta}(t_{1},\dots, t_{p+1})$ in $(p+1)$ variables $t_{1},\dots, t_{p+1}$
with rational coefficients such that

{\em (i)} \,$\Phi_{\eta}(r_{1},\dots, r_{p+q}) = trdeg_{K}K(\D\bigcup_{j=1}^{s} \Lambda(r_{1},\dots, r_{p+1})\eta_{j})$

\noindent for all sufficiently large $(r_{1},\dots,r_{p+1})\in {\bf N}^{p+1}$ (it means that there exist nonnegative integers\, $s_{1},\dots,s_{p+1}$ such that the last equality holds for all $(r_{1},\dots, r_{p+1})\in {\bf N}^{p+1}$ with $r_{1}\geq s_{1}, \dots, r_{p+1}\geq s_{p+1}$);

{\em (ii)} \, $deg_{t_{i}}\Phi_{\eta} \leq m_{i}$ ($1\leq i\leq p$)\,\,\, and $deg_{t_{p+1}}\Phi_{\eta} \leq n$, so that $deg\,\Phi_{\eta}\leq m+n$ and $\Phi_{\eta}(t_{1},\dots, t_{p+1})$ can be represented as
$$\Phi_{\eta}(t_{1},\dots, t_{p+1}) = \D\sum_{i_{1}=0}^{m_{1}}\dots \D\sum_{i_{p}=0}^{m_{p}}\D\sum_{i_{p+1}=0}^{n}a_{i_{1}\dots i_{p+1}}
{t_{1}+i_{1}\choose i_{1}}\dots {t_{p+1}+i_{p+1}\choose i_{p+1}}$$
where $a_{i_{1}\dots i_{p+1}}\in {\bf Z}$ and $2^{n}\,|\,a_{m_{1}\dots m_{p}n}$.

{\em (iii)} \,Let $E_{\eta} = \{(i_{1},\dots, i_{p+1})\in {\bf N}^{p+1}\,|\,0\leq i_{k}\leq m_{k}$ for $k=1,\dots, p$, $0\leq i_{p+1}\leq n$,
and $a_{i_{1}\dots i_{p+1}}\neq 0\}$. Then $d = deg\,\Phi_{\eta}$, $a_{m_{1}\dots m_{p+1}}$, elements $(k_{1},\dots, k_{p+1})\in E_{\eta}'$, the corresponding coefficients $a_{k_{1}\dots k_{p+1}}$ and the coefficients of the terms of total degree $d$ do not depend on the choice of the
system of $\Delta$-$\sigma$-generators $\eta$.
\end{theorem}

\begin{definition}
The polynomial $\Phi_{\eta}(t_{1},\dots, t_{p+1})$ is said to be the difference-differential {\em (or $\Delta$-$\sigma$-)} dimension polynomial of the $\Delta$-$\sigma$-field extension $L/K$ associated with the set of $\Delta$-$\sigma$-generators $\eta$ and partition {\em (3.1)} of the basic set of derivations.
\end{definition}

The $\Delta$-$\sigma$-dimension polynomial associated with partition (3.1) has the following interpretation
as the strength of a system of difference-differential equations.

Let us consider a system of partial difference-differential equations
\begin{equation}
A_{i}(f_{1},\dots, f_{s}) = 0\hspace{0.3in}(i=1,\dots, q)
\end{equation}
over a field of functions of $m$ real variables $x_{1},\dots, x_{m}$ ($f_{1},\dots, f_{s}$ are unknown functions of $x_{1},\dots, x_{m}$). Suppose that $\Delta = \{\delta_{1},\dots, \delta_{m}\}$ where $\delta_{i}$ is the partial differentiation $\partial/\partial x_{i}$ ($i=1,\dots, m$) and the basic set of automorphisms $\sigma = \{\alpha_{1},\dots, \alpha_{m}\}$ consists of $m$ shifts of arguments, $f(x_{1},\dots, x_{m})\mapsto f(x_{1},\dots, x_{i-1}, x_{i}+h_{i}, x_{i+1},\dots, x_{m})$ ($1\leq i\leq m$, $h_{1},\dots, h_{m}\in {\bf R}$). Thus, we assume that the left-hand sides of the equations in (3.2) contain unknown functions $f_{i}$, their partial derivatives, their images under the shifts $\alpha_{j}$, and various compositions of such shifts and partial derivations. Furthermore, we suppose that system (3.2) is algebraic, that is, all $A_{i}(y_{1},\dots, y_{s})$
are elements of a ring of $\Delta$-$\sigma$-polynomials $K\{y_{1},\dots, y_{s}\}$ with coefficients in some functional $\Delta$-$\sigma$-field $K$.

Let us consider a grid with equal cells of dimension $h_{1}\times\dots\times h_{m}$ that fills the whole space ${\bf R}^{m}$. We fix some node $\mathcal{P}$ and say that {\em a node $\mathcal{Q}$ has order $i$} if the shortest path from $\mathcal{P}$ to $\mathcal{Q}$ along the edges of the grid consists of $i$ steps (by a step we mean a path from a node
of the grid to a neighbor node along the edge between them). We also fix partition (3.1) of the set of basic derivations $\Delta$ (such a partition can be, for example,  a natural separation of (all or some) derivations with respect to coordinates and the derivation with respect to time).

For any $r_{1},\dots, r_{p+1}\in {\bf N}$, let us consider the values of the unknown functions $f_{1},\dots, f_{s}$ and their partial derivatives, whose order with respect to $\Delta_{i}$ does not exceed $r_{i}$ ($1\leq i\leq p$), at the nodes whose order does not exceed $r_{p+1}$.  If $f_{1},\dots, f_{s}$ should not satisfy any system of equations (or any other condition), these values can be chosen arbitrarily. Because of the system (and equations obtained from the equations of the system by partial differentiations and transformations of the form $f_{j}(x_{1},\dots, x_{m})\mapsto f_{j}(x_{1}+k_{1}h_{1},\dots, x_{m}+k_{m}h_{m})$ with $k_{1},\dots, k_{m}\in {\bf Z}$, $1\leq j\leq s$), the number of independent values of the functions $f_{1},\dots, f_{s}$ and their partial derivatives whose $i$th order does not exceed $r_{i}$ ($1\leq i\leq p$) at the nodes of order $\leq r_{p+1}$ decreases. This number, which is a function of $p+1$ variables $r_{1},\dots, r_{p+1}$, is the ``measure of strength'' of the system  in the sense of A. Einstein. We denote it by $S_{r_{1},\dots, r_{p+1}}$.

Suppose that the $\Delta$-$\sigma$-ideal $J$ generated in the ring $K\{y_{1},\dots, y_{s}\}$ by the $\Delta$-$\sigma$-polynomials $A_{1},\dots, A_{q}$ is prime (e. g., the polynomials are linear). Then the field of fractions $L$ of the $\Delta$-$\sigma$-integral domain $K\{y_{1},\dots, y_{s}\}/J$ has a natural structure of a $\Delta$-$\sigma$-field extension of $K$ generated by the finite set $\eta = \{\eta_{1},\dots, \eta_{s}\}$ where $\eta_{i}$ is the canonical image of $y_{i}$ in $K\{y_{1},\dots, y_{s}\}/J$ ($1\leq i\leq s$). It is easy to see that the $\Delta$-$\sigma$-dimension polynomial $\Phi_{\eta}(t_{1},\dots, t_{p+1})$ of the extension $L/K$ associated with the system of $\Delta$-$\sigma$-generators $\eta$ has the property that $\Phi_{\eta}(r_{1},\dots, r_{p+1}) = S_{r_{1},\dots, r_{p+1}}$ for all sufficiently large $(r_{1},\dots, r_{p+q})\in {\bf N}^{p+1}$, so this dimension polynomial is the measure of strength of the system of difference-differential equations (3.2) in the sense of A. Einstein.

\section{Numerical polynomials of subsets of ${\bf N}^{m}\times {\bf Z}^{n}$}
\setcounter{equation}{0}
\begin{definition}
A polynomial $f(t_{1}, \dots,t_{p})$ in $p$ variables $t_{1},\dots, t_{p}$
($p\geq 1$) with rational coefficients is called
{\em numerical\/} if $f(r_{1},\dots, r_{p})\in {\bf Z}$
for all sufficiently large $(r_{1}, \dots, r_{p})\in{\bf Z}^{p}$.
\end{definition}
Of course, every polynomial with integer coefficients is numerical.  As an example of a numerical polynomial in $p$ variables with noninteger coefficients ($p\geq 1$) one can consider $\prod_{i=1}^{p}{t_{i}\choose m_{i}}$ \, where $m_{1},\dots, m_{p}\in{\bf N}$. (As usual, ${t\choose k}$ ($k\geq 1$) denotes the polynomial $\frac{t(t-1)\dots (t-k+1)}{k!}$, ${t\choose0} = 1$, and ${t\choose k} = 0$ if $k < 0$.)

\smallskip

The following theorem proved in ~\cite[Chapter 2]{KLMP} gives the ``canonical'' representation of a numerical polynomial in several variables.
\begin{theorem}
Let $f(t_{1},\dots, t_{p})$ be a numerical polynomial in $p$ variables and let $deg_{t_{i}}\, f = m_{i}$ ($m_{1},\dots, m_{p}\in{\bf N}$). Then $f(t_{1},\dots, t_{p})$ can be represented as
\begin{equation}
f(t_{1},\dots t_{p}) =\D\sum_{i_{1}=0}^{m_{1}}\dots \D\sum_{i_{p}=0}^{m_{p}}
{a_{i_{1}\dots i_{p}}}{t_{1}+i_{1}\choose i_{1}}\dots{t_{p}+i_{p}
\choose i_{p}}
\end{equation}
with uniquely defined integer coefficients $a_{i_{1}\dots i_{p}}$.
\end{theorem}
In what follows, we deal with subsets of
${\bf N}^{m}\times {\bf Z}^{n}$ ($m, n\geq 1$) and a fixed partition of the set ${\bf N}_{m} = \{1,\dots , m\}$ into $p$ disjoint subsets ($p\geq 1$):
\begin{equation}
{\bf N}_{m} = N_{1}\bigcup\dots N_{p}
\end{equation}
where $N_{1} = \{1,\dots, m_{1}\}$,\dots,
$N_{p} = \{m_{1}+\dots + m_{p-1}+1,\dots, m\}$ \, ($m_{1}+\dots + m_{p} = m$).
If $a = (a_{1},\dots, a_{m+n})\in {\bf N}^{m}\times {\bf Z}^{n}$ we denote the numbers $\sum_{i=1}^{m_{1}}a_{i}$,
$\sum_{i=m_{1}+1}^{m_{1}+m_{2}}a_{i},\dots, \sum_{i=m_{1}+\dots + m_{p-1} + 1}^{m}a_{i}$,
$\sum_{i=m+1}^{m+n}|a_{i}|$ by $ord_{1}a,\dots, ord_{p+1}a$,

\smallskip

\noindent respectively. Furthermore, we consider the set ${\bf Z}^{n}$ as a union
\begin{equation}
{\bf Z}^{n} = \bigcup_{1\leq j\leq 2^{n}}{\bf Z}_{j}^{(n)}
\end{equation}
where
${\bf Z}_{1}^{(n)}, \dots, {\bf Z}_{2^{n}}^{(n)}$ are all different Cartesian products of $n$ sets each of which is either
${\bf N}$ or ${\bf Z_{-}}=\{a\in {\bf Z}|a\leq 0\}$. We assume that ${\bf Z}_{1}^{(n)} = {\bf N}^{n}$ and call ${\bf Z}_{j}^{(n)}$ the {\em $j$th orthant} of the set ${\bf Z}^{n}$ ($1\leq j\leq 2^{n}$). The set ${\bf N}^{m}\times {\bf Z}^{n}$ is considered as a partially
ordered set with the order $\unlhd$ such that $(e_{1},\dots, e_{m}, f_{1},\dots, f_{n})\unlhd (e'_{1},\dots, e'_{m}, f'_{1},\dots, f'_{n})$ if and only if $(f_{1},\dots, f_{n})$ and $(f'_{1},\dots, f'_{n})$ belong to the same orthant ${\bf Z}_{k}^{(n)}$ and the
$(m+n)$-tuple $(e_{1},\dots, e_{m}, |f_{1}|,\dots, |f_{n}|)$ is less than $(e'_{1},\dots, e'_{m}, |f'_{1}|,\dots, |f'_{n}|)$ with respect
to the product order on ${\bf N}^{m+n}$.

In what follows, for any set $A\subseteq {\bf N}^{m}\times {\bf Z}^{n}$, $W_{A}$ will denote the set of all elements of ${\bf N}^{m}\times {\bf Z}^{n}$ that do not exceed any element of $A$ with respect to the order $\unlhd$.  (Thus, $w\in W_{A}$ if and only if there is no $a\in A$ such that $a\unlhd w$.) Furthermore, for any $r_{1},\dots r_{p+1}\in {\bf N}$, $A(r_{1},\dots r_{p+1})$ will denote the set of all elements $x = (x_{1},\dots, x_{m}, x'_{1},\dots, x'_{n})\in A$ such that $ord_{i}x\leq r_{i}$ ($i=1,\dots, p+1$).

\smallskip

The above notation can be naturally restricted to subsets of ${\bf N}^{m}$. If $E\subseteq {\bf N}^{m}$ and $s_{1},\dots, s_{p}\in {\bf N}$, then $E(s_{1},\dots, s_{p})$ will denote the set $\{e = (e_{1},\dots, e_{m})\in E\,|\,ord_{i}(e_{1},\dots, e_{m},0,\dots, 0)\leq s_{i}$ for $i=1,\dots, p\}$ (\,$(e_{1},\dots, e_{m},0,\dots, 0)$ ends with $n$ zeros; it is treated as a point in ${\bf N}^{m}\times {\bf Z}^{n}$.) Furthermore $V_{E}$ will denote the set of all $m$-tuples $v = (v_{1},\dots , v_{m})\in {\bf N}$ which are not greater than or equal to any $m$-tuple from $E$ with respect to the product order on ${\bf N}^{m}$. (Recall that the product order on  ${\bf N}^{m}$ is a partial order $\leq_P$ on ${\bf N}^{m}$ such that $c =(c_{1}, \dots , c_{m})\leq_{P} c' =(c'_{1}, \dots , c'_{m})$ if and only if $c_{i}\leq c'_{i}$ for all $i=1, \dots , m$.   If $c\leq_{P} c'$ and $c\neq c'$, we write $c<_{P} c'$ ).  Clearly, $v=(v_{1}, \dots , v_{m})\in V_{E}$ if and only if for any element  $(e_{1},\dots , e_{m})\in E$, there exists $i\in {\bf N}, 1\leq i\leq m$, such that $e_{i} > v_{i}$.

The following two theorems proved in ~\cite[Chapter 2]{KLMP} generalize the well-known Kolchin's result on the numerical polynomials associated with subsets of ${\bf N}^{m}$ (see ~\cite[Chapter 0, Lemma 16]{K2}) and give an explicit formula for the numerical polynomials in $p$ variables associated with a finite subset of ${\bf N}^{m}$.

\begin{theorem}
Let $E$ be a subset of ${\bf N}^{m}$ where $m = m_{1} + \dots + m_{p}$ for some nonnegative integers  $m_{1},\dots, m_{p}$ ($p\geq 1$). Then there exists a numerical polynomial $\omega_{E}(t_{1},\dots, t_{p})$ with the following properties:

{\em (i)} \,  $\omega_{E}(r_{1},\dots, r_{p}) = Card\,V_{E}(r_{1},\dots, r_{p})$ for all sufficiently large $(r_{1},\dots, r_{p})\in {\bf N}^{p}$.
{\em (As usual, $Card \, M$ denotes the number of elements of a finite set $M$).}

{\em (ii)} \, $deg_{t_{i}}\omega_{E}\leq m_{i}$ for all $i = 1,\dots, p$.

{\em (iii)} \, $deg\,\omega_{E} = m$ if and only if $E=\emptyset$.
Then  $\omega_{E}(t_{1},\dots, t_{p}) = \D\prod_{i=1}^{p}{t_{i}+m_{i}\choose m_{i}}$.
\end{theorem}

\begin{definition}
The polynomial $\omega_{E}(t_{1},\dots, t_{p})$ is called the dimension polynomial of the set
$E\subseteq {\bf N}^{m}$ associated with the partition
$(m_{1},\dots, m_{p})$ of $m$.
\end{definition}

\begin{theorem}
Let $E = \{e_{1}, \dots, e_{q}\}$ ($q\geq 1$) be a finite subset of ${\bf N}^{m}$
and let a partition {\em (4.2)}  of the set ${\bf N}_{m}$ into $p$ disjoint subsets $N_{1},\dots, N_{p}$ be fixed.
Let $e_{i} = (e_{i1}, \dots, e_{im})$ \, ($1\leq i\leq q$)
and for any $l\in {\bf N}$, $0\leq l\leq q$, let $\Gamma (l,q)$
denote the set of all $l$-element subsets of the set
${\bf N}_{q} = \{1,\dots, q\}$. Furthermore, for any
$\sigma \in \Gamma (l,q)$, let $\bar{e}_{\emptyset j} = 0$, $\bar{e}_{\sigma j} = \max \{e_{ij} |
i\in \sigma\}$ if $\sigma\neq \emptyset$ ($1\leq j\leq m$), and
$b_{\sigma k} =\D\sum_{h\in N_{k}}\bar{e}_{\sigma h}$ ($k = 1,\dots, p$).
Then \begin{equation} \omega_{E}(t_{1},\dots, t_{p}) =
\D\sum_{l=0}^{q}(-1)^{l}\D\sum_{\sigma \in \Gamma (l,\, q)}
\D\prod_{j=1}^{p}{t_{j}+m_{j} - b_{\sigma j}\choose m_{j}}
\end{equation}
\end{theorem}

{\bf Remark.} \, It is clear that if $E\subseteq {\bf N}^{m}$ and $E^{\ast}$ is the set of all minimal elements of the set $E$ with respect to the product order on ${\bf N}^{m}$, then the set $E^{\ast}$ is finite and $\omega_{E}(t_{1}, \dots, t_{p}) = \omega_{E^{\ast}}(t_{1}, \dots, t_{p})$. Thus, Theorem 4.5 gives an algorithm that allows one to find a
numerical polynomial associated with any subset of  ${\bf N}^{m}$  (and with a given partition of the set $\{1,\dots, m\}$): one should first find the set of all minimal points of the subset and then apply Theorem 4.5.

\medskip

The following result can be obtained precisely in the same way as Theorem 3.4 of  ~\cite{Levin2} (the only difference is that the proof in the mentioned paper uses Theorem 3.2 of ~\cite{Levin2} in the case  $p=2$, while the proof of the theorem below should refer to the Theorem 3.2 of ~\cite{Levin2} where $p$ is any positive integer).

\begin{theorem}
Let $A\subseteq {\bf N}^{m}\times{\bf Z}^{n}$ and let partition {\em (4.2)} of ${\bf N}_{m}$ be fixed. Then there exists a numerical polynomial $\phi_{A}(t_{1},\dots, t_{p+1})$ in $p+1$ variables such that

{\em (i)}\, $\phi_{A}(r_{1},\dots, r_{p+1}) = Card\,W_{A}(r_{1},\dots, r_{p+1})$ for all sufficiently large

\noindent$(r_{1},\dots, r_{p+1})\in {\bf N}^{p+1}$.

\medskip

{\em (ii)}\, $deg_{t_{i}}\phi_{A}\leq m_{i}$ ($1\leq i\leq p$), $deg_{t_{p+1}}\phi_{A}\leq n$ and the coefficient of $t_{1}^{m_{1}}\dots t_{p}^{m_{p}}t_{p+1}^{n}$ in $\phi_{A}$ is of the form $\D\frac{2^{n}a}{m_{1}!\dots m_{p}!n!}$\, with $a\in {\bf Z}$.

\medskip

{\em (iii)}\, Let us consider a mapping $\rho: {\bf N}^{m}\times {\bf Z}^{n}\longrightarrow{\bf N}^{m+2n}$
such that

\smallskip

$\rho((e_{1},\dots, e_{m+n}) =(e_{1},\dots, e_{m}, \max \{e_{m+1}, 0\}, \dots, \max \{e_{m+n}, 0\},$

$\max \{-e_{m+1}, 0\}, \dots, \max \{-e_{m+n}, 0\})$.

\smallskip

Let $B = \rho(A)\bigcup \{\bar{e}_{1},\dots, \bar{e}_{n}\}$ where $\bar{e}_{i}$
($1\leq i\leq n$) is a $(m+2n)$-tuple in ${\bf N}^{m+2n}$ whose
$(m+i)$th and $(m+n+i)$th coordinates are equal to 1 and all other
coordinates are equal to 0.
Then $\phi_{A}(t_{1}, \dots, t_{p+1}) = \omega_{B}(t_{1}, \dots, t_{p+1})$ where
$\omega_{B}(t_{1},\dots,  t_{p+1})$ is the dimension polynomial of the set $B$
(see {\em Definition 4.4}) associated with the partition ${\bf N}_{m+2n} =
\{1,\dots , m_{1}\}\bigcup \{m_{1}+1,\dots , m_{1}+m_{2}\}\bigcup\dots
\bigcup \{m_{1}+\dots +m_{p-1}+1,\dots , m\}\bigcup \{m+1,\dots, m+2n\}$ of the
set ${\bf N}_{m+2n}$.

\smallskip

{\em (iv)}\, If $A = \emptyset$, then
\begin{equation}\phi_{A}(t_{1}, \dots, t_{p+1}) ={{t_{1}+m_{1}}\choose m_{1}}\dots {{t_{p}+m_{p}}\choose m_{p}}
\sum_{i=0}^{n}(-1)^{n-i}2^{i}{n\choose i}{{t_{p+1}+i}\choose i}.\end{equation}
\end{theorem}

The polynomial $\phi_{A}(t_{1},\dots, t_{p+1})$ is called the {\em dimension polynomial} of the
set $A\subseteq {\bf N}^{m}\times{\bf Z}^{n}$ associated with partition (4.2) of ${\bf N}_{m}$.

\section{Proof of the main theorem and computation of difference-differential dimension
polynomials via characteristic sets}
\setcounter{equation}{0}
In this section we prove Theorem 3.1 and give a method of computation of difference-differential
dimension polynomials of $\Delta$-$\sigma$-field extensions based on
constructing a characteristic set of the defining prime $\Delta$-$\sigma$-ideal of the
extension.

In what follows we use the conventions  of section 3. In particular, we assume that partition (3.1) of the set of basic derivations $\Delta = \{\delta_{1},\dots, \delta_{m}\}$ is fixed.

\medskip
Let us consider $p+1$ total orderings $<_{1}, \dots, <_{p}, <_{\sigma}$ of the
set of power products $\Lambda$ such that

\medskip

\noindent$\lambda  = \delta_{1}^{k_{1}}\dots \delta_{m}^{k_{m}}
\alpha_{1}^{l_{1}}\dots \alpha_{n}^{l_{n}}
<_{i} \lambda'  = \delta_{1}^{k'_{1}}\dots
\delta_{m}^{k'_{m}}\alpha_{1}^{l'_{1}}\dots\alpha_{n}^{l'_{n}}$  ($1\leq i\leq p$)  if and only if

\medskip

\noindent $(ord_{i}\lambda, ord\,\lambda, ord_{1}\lambda,\dots, ord_{i-1}\lambda, ord_{i+1}\lambda, \dots, ord_{p}\lambda,
ord_{\sigma}\lambda,  k_{m_{1}+\dots + m_{i-1}+1},\dots,$

 \medskip

\noindent$k_{m_{1}+\dots +m_{i}},\, k_{1},\dots, k_{m_{1}+\dots +
m_{i-1}}, k_{m_{1}+\dots + m_{i}+1},\dots, k_{m}, |l_{1}|,\dots, |l_{n}|, l_{1},\dots, l_{n})$ is

\medskip

\noindent less than
$(ord_{i}\lambda', ord\,\lambda', ord_{1}\lambda',\dots, ord_{i-1}\lambda', ord_{i+1}\lambda', \dots, ord_{p}\lambda', ord_{\sigma}\lambda',$

\medskip

\noindent$ k'_{m_{1}+\dots + m_{i-1}+1},\dots, k'_{m_{1}+\dots +m_{i}},\, k'_{1},\dots, k'_{m_{1}+\dots +
m_{i-1}}, k'_{m_{1}+\dots + m_{i}+1},\dots,$

\medskip

\noindent$k'_{m}, |l'_{1}|,\dots, |l'_{n}|, l'_{1},\dots, l'_{n})$ with respect to the lexicographic order on ${\bf N}^{m+2n+p+2}$.

\medskip

Similarly, $\lambda <_{\sigma} \lambda'$ if and only if $(ord_{\sigma}\lambda, ord\,\lambda, ord_{1}\lambda,\dots, ord_{p}\lambda,
|l_{1}|,\dots, |l_{n}|,$

\medskip

\noindent$l_{1},\dots, l_{n}, k_{1},\dots, k_{m})$ is less than the corresponding $(m+2n+p+2)$-tuple for $\lambda'$
with respect to the lexicographic order on ${\bf N}^{m+2n+p+2}$.

\medskip

Two elements $\lambda_{1} = \delta_{1}^{k_{1}}\dots \delta_{m}^{k_{m}}
\alpha_{1}^{l_{1}}\dots \alpha_{n}^{l_{n}}$ and
$\lambda_{2} = \delta_{1}^{r_{1}}\dots \delta_{m}^{r_{m}}\alpha_{1}^{s_{1}}
\dots \alpha_{n}^{s_{n}}$ in $\Lambda$ are called {\em similar\/}, if
the $n$-tuples $(l_{1}, \dots, l_{n})$ and $(s_{1}, \dots, s_{n})$
belong to the same orthant of ${\bf Z}^{n}$ (see (4.3)\,). In this case we
write $\lambda_{1}\sim \lambda_{2}$. We say that $\lambda_{1}$ {\em divides} $\lambda_{2}$
(or $\lambda_{2}$ is a {\em multiple} of $\lambda_{1}$) and write $\lambda_{1}|\lambda_{2}$
if $\lambda_{1}\sim \lambda_{2}$ and there exists
$\lambda \in \Lambda$ such that $\lambda\sim \lambda_{1}, \,\lambda\sim \lambda_{2}$ and $\lambda_{2} = \lambda\lambda_{1}$.

\smallskip

Let $K$ be a difference-differential field ($Char\,K = 0$) with the basic sets $\Delta$ and $\sigma$ described above
(and partition (3.1) of the set $\Delta$). Let $K\{y_{1},\dots, y_{s}\}$ be the ring of
$\Delta$-$\sigma$-polynomials over $K$ and let $\Lambda Y$ denote
the set of all elements $\lambda y_{i}$ ($\lambda\in \Lambda$, $1\leq i\leq s$) called {\em terms}.
Note that as a ring, $K\{y_{1},\dots, y_{s}\} = K[\Lambda Y]$. Two terms $u=\lambda y_{i}$ and $v=\lambda' y_{j}$
are called {\em similar} if $\lambda$ and $\lambda'$ are similar; in this case we write $u\sim v$.
If $u = \lambda y_{i}$ is a term and $\lambda'\in \Lambda$, we say that $u$ is similar to $\lambda'$ and write
$u\sim \lambda'$ if $\lambda\sim \lambda'$. Furthermore, if $u, v\in \Lambda Y$, we say that $u$ {\em divides} $v$ or
{\em $v$ is a multiple of $u$}, if $u=\lambda' y_{i}$, $v=\lambda'' y_{i}$ for some $y_{i}$ and $\lambda'|\lambda''$.
(If $\lambda'' = \lambda\lambda'$ for some $\lambda\in \Lambda, \,\lambda\sim \lambda'$, we write $\D\frac{v}{u}$ for $\lambda$.)

\smallskip

Let us consider $p+1$ orders $<_{1},\dots, <_{p}, <_{\sigma}$ on the set $\Lambda Y$ that
correspond to the orders on the semigroup $\Lambda$ (we use the same symbols for the orders on $\Lambda$
and $\Lambda Y$). These orders  are defined as follows: $\lambda y_{j} <_{i}$ (or $<_{\sigma}$) $\lambda' y_{k}$ if and only if
$\lambda <_{i}$ (respectively, $<_{\sigma}$)$\lambda'$ in $\Lambda$ or $\lambda = \lambda'$ and $j < k$ ($1\leq i\leq p,\, 1\leq j, k\leq s$).

\smallskip

The order of a term $u = \lambda y_{k}$ and its orders with respect to the sets $\Delta_{i}$ ($1\leq i\leq p$) and $\sigma$ are defined as the corresponding orders of $\lambda$ (we use the same notation $ord\,u$, $ord_{i}u$, and $ord_{\sigma}u$ for the corresponding orders).

\smallskip

If $A\in K\{y_{1},\dots, y_{s}\}\setminus K$ and $1\leq k\leq p$, then the highest with respect to $<_{k}$ term that appears in $A$ is called the {\em $k$-leader} of $A$. It is denoted by $u_{A}^{(k)}$. The highest term of $A$ with respect to $<_{\sigma}$ is called the {\em $\sigma$-leader} of $A$; it is denoted by $v_{A}$. If $A$ is written as a polynomial in  $v_{A}$, $A = I_{d}{(v_{A})}^{d} + I_{d-1}{(v_{A})}^{d-1} + \dots + I_{0}$, where all terms of $I_{0},\dots, I_{d}$ are less than $v_{A}$ with respect to $<_{\sigma}$, then $I_{d}$ is called the {\em initial\/} of $A$. The partial derivative $\partial A/\partial v_{A} = dI_{d}(v_{A})^{d-1} + (d-1)I_{d-1}{(v_{A})}^{d-2} + \dots + I_{1}$ is called the {\em separant\/} of $A$. The initial and the separant of a $\Delta$-$\sigma$-polynomial $A$ are denoted by $I_{A}$ and $S_{A}$, respectively.

\smallskip

If $A, B\in K\{y_{1},\dots, y_{s}\}$, then $A$ is said to have lower rank than $B$ (we write $rk\,A < rk\,B$) if either $A\in K$, $B\notin K$, or $(v_{A}, deg_{v_{A}}A, ord_{1}u_{A}^{(1)},\dots, ord_{p}u_{A}^{(p)})$ is less than $(v_{B}, deg_{v_{B}}B, ord_{1}u_{B}^{(1)},\dots, ord_{p}u_{B}^{(p)})$ with respect to the lexicographic order ($v_{A}$ and $v_{B}$ are compared with respect to $<_{\sigma}$). If the vectors are equal (or $A, B\in K$) we say that $A$ and $B$ are of the same rank and write $rk\,A = rk\, B$.

\begin{definition}
If $A, B\in K\{y_{1},\dots, y_{s}\}$, then $B$ is said to be
reduced with respect to $A$ if

{\em (i)} $B$ does not contain terms $\lambda v_{A}$ such that $\lambda\sim v_{A}$,
$\lambda_{\Delta}\neq 1$, and  $ord_{i}(\lambda u_{A}^{(i)})\leq ord_{i}u_{B}^{(i)}$
for $i=1,\dots, p$.

{\em (ii)} If $B$ contains a term $\lambda v_{A}$, where $\lambda\sim v_{A}$ and
$\lambda_{\Delta} = 1$, then either there exists
$j, \,  1\leq j\leq p$, such that  $ord_{j} u_{B}^{(j)} < ord_{j}(\lambda u_{A}^{(j)})$
or $ord_{j}(\lambda u_{A}^{(j)}) \leq ord_{j} u_{B}^{(j)}$
for all $j = 1,\dots, p$ and $deg_{\lambda v_{A}}B < deg_{v_{A}}A$.
\end{definition}

If $B\in K\{y_{1},\dots, y_{s}\}$, then $B$ is said to be {\em reduced with respect to a set\, $\Sigma \subseteq K\{y_{1},\dots, y_{s}\}$} if $B$ is reduced with respect to every element of $\Sigma$.

\medskip

A set  $\Sigma \subseteq K\{y_{1},\dots, y_{s}\}$ is called {\em autoreduced} if $\Sigma \bigcap K = \emptyset$ and every element of $\Sigma$ is reduced with respect to any other element of this set.

The proof of the following lemma can be found in ~\cite[Chapter0, Section 17]{K2}.
\begin{lemma}
Let $A$ be any infinite subset of ${\bf N}^{m}\times{\bf N}_{n}$ ($m,n\in {\bf N}$, $n\geq 1$). Then there exists an infinite sequence of elements of $A$, strictly increasing relative to the product order, in which every element has the same projection on ${\bf N}_{n}$.
\end{lemma}
This lemma implies the following statement that will be used below.
\begin{lemma}
Let $S$ be any infinite set of terms $\lambda y_j$ ($\lambda\in \Lambda, 1\leq j\leq s$) in $K\{y_{1},\dots, y_{s}\}$.   Then there exists an index $j$ ($1\leq j\leq s$) and an infinite sequence of terms $\lambda_{1}y_{j}, \lambda_{2}y_{j}, \dots, \lambda_{k}y_{j},\dots $ such that $\lambda_{k}|\lambda_{k+1}$ for every $k=1, 2, \dots $.
\end{lemma}

\begin{proposition}
Every autoreduced set is finite.
\end{proposition}

\begin{proof}
Suppose that $\Sigma$ is an infinite autoreduced subset of $K\{y_{1},\dots, y_{s}\}$. Then  $\Sigma$ must contain an infinite set  $\Sigma'\subseteq \Sigma$ such that all $\Delta$-$\sigma$-polynomials from $\Sigma'$ have different $\sigma$-leaders similar to each other. Indeed, if it is not so, then there exists an infinite set $\Sigma_{1}\subseteq \Sigma$ such that all $\Delta$-$\sigma$-polynomials from $\Sigma_{1}$ have the same $\sigma$-leader $v$. By Lemma 5.2, the infinite set $\{(ord_{1}u_{A}^{(1)}, \dots, ord_{p}u_{A}^{(p)}) | A\in \Sigma_{1}\}$ contains a nondecreasing infinite sequence $$(ord_{1}u_{A_{1}}^{(1)}, \dots, ord_{p}u_{A_{1}}^{(p)}) \leq_{P}
(ord_{1}u_{A_{2}}^{(1)}, \dots, ord_{p}u_{A_{2}}^{(p)}) \leq_{P} \dots $$ ($A_{1}, A_{2}, \dots \in \Sigma_{1}$ and $\leq_{P}$ denotes the product order on ${\bf N}^{p}$). Since the sequence $\{deg_{v_{A_{i}}}A_{i} | i = 1, 2, \dots \}$ cannot be strictly decreasing, there are two indices $i$ and $j$ such that $i < j$ and $deg_{v_{A_{i}}}A_{i} \leq deg_{v_{A_{j}}}A_{j}$. We see that $A_{j}$ is not reduced with respect to  $A_{i}$ that contradicts the fact that $\Sigma$ is an autoreduced set.

Thus, we can assume that all $\Delta$-$\sigma$-polynomials of our infinite autoreduced set $\Sigma$ have distinct $\sigma$-leaders similar to each other. Using Lemma 5.3, we can assume that there exists an infinite  sequence $B_{1}, B_{2}, \dots $ \, of elements of $\Sigma$ such that
$v_{B_{i}} | v_{B_{i+1}}$  and $\left(\frac{v_{B_{i+1}}}{v_{B_{i}}}\right)_{\Delta} \neq 1$ for all $i = 1, 2, \dots $.
Let $k_{ij} = ord_{j}v_{B_{i}}$ and  $l_{ij} = ord_{j}u_{B_{i}}^{(j)}$ ($1\leq j\leq p$). Obviously, $l_{ij}\geq k_{ij}$ ($i = 1, 2,\dots ; j = 1,\dots, p$),
so that $\{(l_{i1}-k_{i1}, \dots, l_{ip}-k_{ip}) | i =1, 2, \dots \}\subseteq {\bf N}^{p}$. By Lemma 5.2, there exists an infinite sequence of indices $i_{1} < i_{2} < \dots $ such that $(l_{i_{1}1}-k_{i_{1}1},\dots, l_{i_{1}p}-k_{i_{1}p}) \leq_{P} (l_{i_{2}1}-k_{i_{2}1},\dots, l_{i_{2}p}-k_{i_{2}p}) \leq_{P} \dots $.
Then for any $j = 1,\dots, p$, we have $ord_{j}\,(\, \frac{v_{B_{i_{2}}}}{v_{B_{i_{1}}}}u_{B_{i_{1}}}^{(j)}) = k_{i_{2}j} - k_{i_{1}j} + l_{i_{1}j} \leq
k_{i_{2}j} + l_{i_{2}j} - k_{i_{2}j} = l_{i_{2}j} = ord_{j}u_{B_{i_{2}}}^{(j)}$, so that $B_{i_{2}}$ contains a term $\lambda v_{B_{i_{1}}} =  v_{B_{i_{2}}}$ such that  $\lambda_{\Delta}\neq 1$ and $ord_{j}(\lambda u_{B_{i_{1}}}^{(j)}) \leq ord_{j} u_{B_{i_{2}}}^{(j)}$ for $j = 1, \dots, p$. Thus, the $\Delta$-$\sigma$-polynomial $B_{i_{2}}$ is reduced with respect to $B_{i_{1}}$ that contradicts the fact that $\Sigma$ is an autoreduced set.
\end{proof}

Throughout the rest of the paper, while considering autoreduced sets in the ring $K\{y_{1},\dots, y_{s}\}$ we always assume that their elements are arranged in order of increasing rank. (Therefore, if we consider an autoreduced set of $\Delta$-$\sigma$-polynomials $\Sigma = \{A_{1},\dots, A_{d}\}$, then $rk\,A_{1}< \dots < rk\,A_{d}$).

\begin{proposition} Let $\Sigma = \{A_{1},\dots, A_{d}\}$ be an autoreduced set in the ring $K\{y_{1},\dots, y_{s}\}$ and let $I_{k}$ and $S_{k}$ ($1\leq k\leq d$) denote the initial and separant of $A_{k}$, respectively. Furthermore, let $I(\Sigma) = \{X\in K\{y_{1},\dots, y_{s}\}\,|\,X=1$ or $X$ is a product of finitely many elements of the form $\gamma(I_{k})$ and $\gamma'(S_{k})$ where $\gamma, \gamma'\in \Lambda_{\sigma}\}$. Then for any $\Delta$-$\sigma$-polynomial $B$, there exist $B_{0}\in K\{y_{1},\dots, y_{s}\}$ and $J\in I(\Sigma)$
such that $B_{0}$ is reduced with respect to $\Sigma$ and $JB\equiv B_{0} \,(mod [\Sigma])$ (that is, $JB-B_{0}\in [\Sigma]$).
\end{proposition}

\begin{proof}
If $B$ is reduced with respect to $\Sigma$, the statement is obvious (one can set $B_{0}=B$). Suppose that $B$ is not reduced with respect to $\Sigma$.
Let $u^{(j)}_i$ and $v_i$ ($1\leq j\leq p,\, 1\leq i\leq d$) be the leaders of the element $A_i$ relative to the orders $<_{j}$ and $<_{\sigma}$, respectively.
In what follows, a term $w_H$, that appears in a $\Delta$-$\sigma$-polynomial $H\in R$, will be called a $\Sigma$-leader of $H$ if $w_H$ is the greatest (with respect to
$<_{\sigma}$) term among all terms $\lambda v_i$ ($\lambda\in\Lambda$,  $1\leq j\leq d$) such that $\lambda\sim v_{i}$, $\lambda v_{i}$ appears in $H$ and either $\lambda_{\Delta}\neq 1$ and $ord_{j}(\lambda u^{(j)}_{i})\leq ord_{j}u^{(j)}_H$ for $j=1,\dots, p$, or $\lambda_{\Delta} = 1$, $ord_{j}(\lambda u^{(j)}_{i})\leq ord_{j}u^{(j)}_H$ ($1\leq j\leq p$),
and $deg_{v_{i}}A_{i} \leq deg_{\lambda v_{i}}H$.

Let $w_B$ be the $\Sigma$-leader of $B$. Then $B = B'w_{B}^{r} + B''$ where $B'$ does not contain $w_B$ and $deg_{w_{B}}B'' < r$. Let $w_{B}=\lambda v_i$ for some $i$ ($1\leq i\leq d$) and for some $\lambda \in \Lambda$, $\lambda\sim v_i$, such that $ord_{j}(\lambda u^{(j)}_{i})\leq ord_{j}u^{(j)}_B$ for $j=1,\dots, p$. Without loss of generality we may assume that $i$ corresponds to the maximum (with respect to $<_\sigma$) $\sigma$-leader $v_i$ in the set of all $\sigma$-leaders of elements of $\Sigma$.

Suppose, first, that $\lambda_{\Delta}\neq 1$ (and $ord_{j}(\lambda u^{(j)}_{i})\leq ord_{j}u^{(j)}_B$ for $j=1,\dots, p$).
Then $\lambda_{\Delta}A_{i} - S_{i}\lambda_{\Delta} v_{i}$ has lower rank than $\lambda_{\Delta} v_{i}$, hence
$T = \lambda A_{i} - \lambda_{\sigma}(S_{i})\lambda v_{i}$ has lower rank than $\lambda v_{i} = w_{B}$. Also,
$(\lambda_{\sigma}(S_{i}))^{r}B = (\lambda_{\sigma}(S_{i})\lambda v_{i})^{r}B' + (\lambda_{\sigma}(S_{i}))^{r}B'' =
(\lambda A_{i} - T)^{r}B' + (\lambda_{\sigma}(S_{i}))^{r}B''$. Setting $B^{(1)} = B'(-T)^{r} + (\lambda_{\sigma}(S_{i}))^{r}B''$ we
obtain that $B^{(1)}\equiv B\, (mod [\Sigma])$, $B^{ (1)}$
does not contain any $\Sigma$-leader, which is greater than $w_{B}$ with respect to $<_{\sigma}$, and $deg_{w_{B}} B^{(1)} < r$.

Now let $\lambda_{\Delta} = 1$, $ord_{j}(\lambda u^{(j)}_{i})\leq ord_{j}u^{(j)}_B$ ($1\leq j\leq p$), and $r_{i} < r$ where $r_{i} = deg_{v_{i}}A_{i}$. Then the $\Delta$-$\sigma$-polynomial $(\lambda I_{i})B - w_{B}^{r-r_{i}}(\lambda A_{i})B'$ has all the properties of $B^{(1)}$ mentioned above. Repeating the described procedure,
we arrive at a desired $\Delta$-$\sigma$-polynomial $B_{0}$, which is reduced with respect to $\Sigma$ and satisfies the condition
$JB\equiv B_{0} \,(mod [\Sigma])$, where $J=1$ or $J$ is a product of finitely many elements of the form $\gamma(I_{k})$
and $\gamma'(S_{k})$  ($\gamma, \gamma'\in \Lambda_{\sigma}$).
\end{proof}

With the notation of the last proposition, we say that the $\Delta$-$\sigma$-polynomial $B$
{\em reduces to $B_{0}$} modulo $\Sigma$.

\begin{definition}
Let  $\Sigma = \{A_{1},\dots,A_{d}\}$ and
$\Sigma' = \{B_{1},\dots,B_{e}\}$ be two autoreduced sets in the ring of
$\Delta$-$\sigma$-polynomials  $K\{y_{1},\dots, y_{s}\}$. An autoreduced set
$\Sigma$ is said to have lower rank than $\Sigma'$ if one of the following
two cases holds:

{\em (1)} There exists $k\in {\bf N}$ such that $k\leq \min \{d,e\}$,
$rk\,A_{i}=rk\,B_{i}$ for $i=1,\dots,k-1$ and  $rk\,A_{k} < rk\,B_{k}$.

{\em (2)} $d>e$ and  $rk\,A_{i}=rk\,B_{i}$ for $i=1,\dots,e$.

If $d=e$ and $rk\,A_{i}=rk\,B_{i}$ for $i=1,\dots,d$, then $\Sigma$ is
said to have the same rank as $\Sigma'$.
\end{definition}

\begin{proposition}
In every nonempty family of autoreduced sets of difference-differential polynomials
there exists an autoreduced set of lowest rank.
\end{proposition}

\begin{proof}
Let $\Phi$ be a nonempty family of autoreduced sets in the ring  $K\{y_{1},\dots, y_{s}\}$. Let us inductively define an infinite descending chain of subsets of $\Phi$ as follows:
$\Phi_{0}=\Phi$, $\Phi_{1}=\{\Sigma \in \Phi_{0} | \Sigma$ contains at least one element and the first element of $\Sigma$ is of lowest possible rank\}, \dots , $\Phi_{k}=\{\Sigma \in \Phi_{k-1} | \Sigma$ contains at least $k$ elements and the $k$th element of $\Sigma$ is of lowest possible rank\}, \dots . It is clear that if $A$ and $B$ are any two $\Delta$-$\sigma$-polynomials in the same set  $\Phi_{k}$, then $v_{A} = v_{B}$, $deg_{v_{A}}A =  deg_{v_{B}}B$, and $ord_{i}u_{A}^{(i)} = ord_{i}u_{B}^{(i)}$ for $i = 1,\dots, p$.
Therefore, if all sets $\Phi_{k}$ are nonempty, then the set \{$A_{k}|A_{k}$ is the $k$th element of some autoreduced set in $\Phi_{k}$\} would be an infinite autoreduced set, and this would contradict Proposition 5.4. Thus, there is the smallest positive integer $k$ such that $\Phi_{k}=\emptyset$. Clearly, every element of $\Phi_{k-1}$ is an autoreduced set  of
lowest rank in $\Phi$.
\end{proof}

Let $J$ be any ideal of the ring   $K\{y_{1},\dots, y_{s}\}$. Since the set of all autoreduced subsets of $J$ is not empty (if $A\in J$, then $\{A\}$ is an autoreduced subset of $J$), the last statement shows that  $J$ contains an autoreduced subset of lowest rank. Such an autoreduced set is called a {\em characteristic set} of the ideal $J$.

\begin{proposition}
Let $\Sigma = \{A_{1}, \dots , A_{d}\}$ be a characteristic set of a $\Delta$-$\sigma$-ideal
$J$ of the ring  $R = K\{y_{1},\dots, y_{s}\}$.  Then
an element $B\in R$ is reduced with respect to the set $\Sigma$ if and only
if $B = 0$.
\end{proposition}

\begin{proof}  First of all, note that if $B\neq 0$ and $rk\,B < rk\,A_{1}$, then $rk\,\{B\} < rk\,\Sigma$ that contradicts the fact that $\Sigma$ is a characteristic set of the ideal
$J$. Let $rk\,B > rk\,A_{1}$ and let $A_{1},\dots, A_{j}$ ($1\leq j\leq d$) be all elements of $\Sigma$ whose rank is lower that the rank of $B$. Then $\Sigma' = \{A_{1},\dots, A_{j}, B\}$ is an autoreduced set of lower rank than $\Sigma$, contrary to the fact that $\Sigma$ is a characteristic set of $J$. Thus, $B = 0$.
\end{proof}
Since for any $\Delta$-$\sigma$-polynomial $A$ and any $\gamma\in \Lambda_{\sigma}$, $ord_{i}(\gamma A) = ord_{i}A$ for $i=1,\dots, p$, one can introduce the concept of a coherent autoreduced set of a linear $\Delta$-$\sigma$-ideal of $K\{y_{1},\dots, y_{s}\}$ (that is, a $\Delta$-$\sigma$-ideal generated by a finite set of linear $\Delta$-$\sigma$-polynomials) in the same way as it is defined in the case of difference polynomials (see ~\cite[Section 6.5]{KLMP}): an autoreduced set $\Sigma = \{A_{1},\dots, A_{d}\}\subseteq K\{y_{1},\dots, y_{s}\}$ consisting of linear $\Delta$-$\sigma$-polynomials is called {\em coherent} if it satisfies the following two conditions:

(i)\, $\lambda A_{i}$ reduces to zero modulo $\Sigma$ for any $\lambda\in \Lambda, \,1\leq i\leq d$.

(ii)\, If $v_{A_{i}}\sim v_{A_{j}}$ and $w = \lambda v_{A_{i}} = \lambda'v_{A_{j}}$, where $\lambda\sim\lambda'\sim v_{A_{i}}\sim v_{A_{j}}$,
then the $\Delta$-$\sigma$-polynomial $(\lambda'I_{A_{j}})(\lambda A_{i}) - (\lambda I_{A_{i}})(\lambda'A_{j})$ reduces to zero modulo $\Sigma$.

The following two propositions can be proved precisely in the same way as the corresponding statements for difference polynomials,
see ~\cite[Theorem 6.5.3 and Corollary 6.5.4]{KLMP}).

\begin{proposition}
Any characteristic set of a linear $\Delta$-$\sigma$-ideal of the ring of $\Delta$-$\sigma$-polynomials $K\{y_{1},\dots, y_{s}\}$
is a coherent autoreduced set. Conversely, if $\Sigma$ is a coherent autoreduced set in $K\{y_{1},\dots, y_{s}\}$ consisting of linear
$\Delta$-$\sigma$-polynomials, then $\Sigma$ is a characteristic set of the linear $\Delta$-$\sigma$-ideal $[\Sigma]$.
\end{proposition}
\begin{proposition}
Let us consider a partial order $\preccurlyeq$ on $K\{y_{1},\dots, y_{s}\}$ such that $A\preccurlyeq B$ if and only if $v_{A}|v_{B}$. Let $A$
be a linear $\Delta$-$\sigma$-polynomial in $K\{y_{1},\dots, y_{s}\}$, $A\notin K$. Then the set of all minimal with respect to $\preccurlyeq$
elements of the set $\{\lambda A\,|\,\lambda\in\Lambda\}$ is a characteristic set of the $\Delta$-$\sigma$-ideal $[A]$.
\end{proposition}

Now we are ready to prove Theorem 3.1.

\begin{proof}
Let $L = K\langle\eta_{1},\dots, \eta_{s}\rangle$ be a $\Delta$-$\sigma$-field extension of $K$ generated by a finite set
$\eta = \{\eta_{1},\dots, \eta_{s}\}$.  Then there exists a natural $\Delta$-$\sigma$-homomorphism $\Upsilon_{\eta}$ of the ring of $\Delta$-$\sigma$-polynomials
$K\{y_{1},\dots, y_{s}\}$ onto the $\Delta$-$\sigma$-subring $K\{\eta_{1},\dots,\eta_{s}\}$ of $L$ such that $\Upsilon_{\eta}(a) = a$ for any
$a\in K$ and $\Upsilon_{\eta}(y_{j}) = \eta_{j}$ for $j = 1,\dots, s$. (If $A\in K\{y_{1},\dots, y_{s}\}$, then $\Upsilon_{\eta}(A)$ is called the {\em value\/} of $A$ at $\eta$;
it is denoted by $A(\eta)$.) Obviously, the kernel $P$ of the $\Delta$-$\sigma$-homomorphism  $\Upsilon_{\eta}$ is a prime $\Delta$-$\sigma$-ideal of
$K\{y_{1},\dots, y_{s}\}$. This ideal is called the {\em defining\/} ideal of $\eta$ over $K$ or the defining ideal of the extension $L = K\langle \eta_{1},\dots,\eta_{s}\rangle$.  It is easy to see that if the quotient field $Q$ of the factor ring $\bar{R} = K\{y_{1},\dots, y_{s}\}/P$ is considered as a $\Delta$-$\sigma$-field (where $\delta(\frac{f}{g}) =
\frac{g\delta(f)-f\delta(g)}{g^2}$ and $\tau(\frac{f}{g}) = \frac{\tau(f)}{\tau(g)}$ for any $f, g\in \bar{R}$,
$\delta\in \Delta,\, \tau\in \sigma^{\ast}$), then $Q$ is naturally $\Delta$-$\sigma$-isomorphic to the field $L$. The corresponding isomorphism  is identity on $K$ and maps
the images of the $\Delta$-$\sigma$-indeterminates $y_{1},\dots, y_{s}$ in the factor ring $\bar{R}$ to the elements $\eta_{1},\dots, \eta_{s}$, respectively.

Let  $\Sigma = \{A_{1},\dots, A_{d}\}$ be a characteristic set of the defining $\Delta$-$\sigma$-ideal $P$.
For any $r_{1},\dots, r_{p+1}\in {\bf N}$, let us set $U_{r_{1}\dots r_{p+1}} =
\{u\in \Lambda Y| ord_{i}u\leq r_{i}$ for $i = 1,\dots, p$, $ord_{\sigma}u\leq r_{p+1}$,  and either $u$ is
not a multiple of any $v_{A_{i}}$ or for every $\lambda\in\Lambda, A\in \Sigma$ such that
$u = \lambda v_{A}$ and $\lambda\sim v_{A}$, there exists $j\in \{1,\dots, p\}$
such that $ord_{j}(\lambda u_{A}^{(j)}) > r_{j}\}$.
We are going to show that the set $\bar{U}_{r_{1}\dots r_{p+1}} =
\{u(\eta)| u\in U_{r_{1}\dots r_{p+1}}\}$ is a transcendence basis of the
field $K(\D\bigcup_{j=1}^{n} \Lambda(r_{1},\dots, r_{p+1})\eta_{j})$ over $K$.

Let us show first that the set  $\bar{U}_{r_{1}\dots r_{p+1}}$ is algebraically independent over $K$. Let $g$ be a polynomial in $k$
variables ($k\in {\bf N},  k \geq 1$) such that $g(u_{1}(\eta),\dots, u_{k}(\eta)) = 0$ for some $u_{1},\dots, u_{k}\in U_{r_{1}\dots r_{p+1}}$. Then the $\Delta$-$\sigma$-polynomial $\bar{g} = g(u_{1},\dots, u_{k})$ is reduced with respect to $\Sigma$. (Indeed, if $g$ contains a term $u = \lambda v_{A_{i}}$  with $\lambda\in\Lambda,\, \lambda\sim v_{A_{i}}$ ($1\leq i\leq d$), then there exists $k\in \{1,\dots, p\}$ such that $ord_{k}(\lambda u_{A_{i}}^{(k)}) > r_{k}\geq ord_{k}u_{\bar{g}}^{(k)}$). Since $\bar{g}\in P$, Proposition 5.8  implies that $\bar{g} = 0$. Thus, the set $\bar{U}_{r_{1}\dots r_{p+1}}$ is algebraically independent over $K$.

Now, let us prove that every element $\lambda \eta_{j}$ ($1\leq j\leq s, \lambda \in \Lambda(r_{1},\dots, r_{p+1})$) is algebraic over the field $K(\bar{U}_{r_{1},\dots, r_{p+1}})$. Let $\lambda \eta_{j} \notin \bar{U}_{r_{1},\dots, r_{p+1}}$ (if $\lambda \eta_{j} \in \bar{U}_{r_{1},\dots, r_{p+1}}$, the statement is obvious). Then  $\lambda y_{j} \notin U_{r_{1},\dots, r_{p+1}}$ whence  $\lambda y_{j}$ is equal to some term  $\lambda'v_{A_{i}}$ where $\lambda'\in\Lambda$, $\lambda\sim v_{A_{i}}$ ($1\leq i\leq d$), and $ord_{k}(\lambda'u_{A_{i}}^{(k)})\leq r_{k}$ for $k = 1, \dots, p$. Let us represent $A_{i}$ as a polynomial in $v_{A_{i}}$: $A_{i} = I_{0}{(v_{A_{i}})}^{e} + I_{1}{(v_{A_{i}})}^{e-1} +\dots +
I_{e}$, where $I_{0}, I_{1},\dots I_{e}$ do not contain $v_{A_{i}}$ (therefore, all terms in these $\Delta$-$\sigma$-polynomials are lower than
$v_{A_{i}}$ with respect to  $<_{\sigma}$). Since $A_{i}\in P$,
\begin{equation}
A_{i}(\eta) =  I_{0}(\eta){(v_{A_{i}})(\eta)}^{e} +
I_{1}(\eta){(v_{A_{i}})(\eta)}^{e-1} +\dots + I_{e}(\eta) = 0
\end{equation}
It is easy to see that  the $\Delta$-$\sigma$-polynomials $I_{0}$ and $S_{A_{i}} = \partial A_{i}/\partial v_{A_{i}}$ are reduced with
respect to any element of the set $\Sigma$. Applying Proposition 5.8 we obtain that $I_{0}\notin P$ and $S_{A_{i}}\notin P$ whence $I_{0}(\eta) \neq 0$ and $S_{A_{i}}(\eta)\neq 0$. Now, if we apply $\lambda'$ to both sides of equation (5.1), the resulting equation will show that the element $\lambda'v_{A_{i}}(\eta) = \lambda\eta_{j}$ is algebraic over the field $K(\{\bar{\lambda}\eta_{l} | ord_{i}\bar{\lambda}\leq r_{i},\, ord_{\sigma}\bar{\lambda}\leq r_{p+1}$, for $i=1,\dots, p,  1\leq l\leq s$, and $\bar{\lambda}y_{l} <_{1} \lambda'u_{A_{i}}^{(1)} = \lambda y_{j}\})$.
Now, the induction on the set of terms $\Lambda Y$ ordered by $<_{\sigma}$ completes the proof of the fact that $\bar{U}_{r_{1}\dots r_{p+1}}(\eta)$ is a transcendence basis of the field $K(\D\bigcup_{j=1}^{s} \Lambda(r_{1},\dots, r_{p+1})\eta_{j})$ over $K$.

\smallskip

Let $U_{r_{1}\dots r_{p+1}}^{(1)} = \{u\in \Lambda Y | ord_{i}u\leq r_{i}$ for
$i = 1,\dots, p$, $ord_{\sigma}u\leq r_{p+1}$, and $u$ is not a multiple of any $v_{A_{j}}$, $j = 1,\dots, d\}$ and
let $U_{r_{1}\dots r_{p+1}}^{(2)} = \{u\in \Lambda Y | ord_{i}u\leq r_{i}$,  $ord_{\sigma}u\leq r_{p+1}$ for
$i = 1,\dots, p$ and there exists at least one pair $i, j$ ($1\leq i\leq p,
1\leq j\leq d$) such that $u = \lambda v_{A_{j}}$, $\lambda\sim v_{A_{j}}$,  and
$ord_{i}(\lambda u_{A_{j}}^{(i)}) > r_{i}\}$.  Clearly, $U_{r_{1}\dots r_{p+1}} =
U_{r_{1}\dots r_{p+1}}^{(1)}\bigcup U_{r_{1}\dots r_{p+1}}^{(2)}$ and
$U_{r_{1}\dots r_{p+1}}^{(1)}\bigcap U_{r_{1}\dots r_{p+1}}^{(2)} = \emptyset$.

By Theorem 4.6, there exists a numerical polynomial $\phi(t_{1},\dots, t_{p+1})$ in $p+1$ variables $t_{1},\dots, t_{p+1}$ such that $\phi(r_{1},\dots, r_{p+1}) = Card\,U_{r_{1}\dots r_{p+1}}^{(1)}$ for all sufficiently large $(r_{1},\dots, r_{p+1})\in {\bf N}^{p+1}$, $deg_{t_{i}}\phi \leq m_{i}$ ($1\leq i\leq p$), and $deg_{t_{p+1}}\phi\leq n$. Furthermore, repeating the arguments of the proof of theorem 4.1 of \cite{Levin3}, we obtain that there is a linear combination $\psi(t_{1},\dots, t_{p+1})$ of polynomials of the form (4.5) such that $\psi(r_{1},\dots, r_{p+1}) = Card\,U_{r_{1}\dots r_{p+1}}^{(2)}$ for all sufficiently large $(r_{1},\dots, r_{p+1})\in {\bf N}^{p+1}$. Then the polynomial
$\Phi_{\eta}(t_{1},\dots, t_{p+1}) = \phi(t_{1},\dots, t_{p+1}) + \psi(t_{1},\dots, t_{p+1})$ satisfies conditions (i) and (ii) of Theorem 3.1.

In order to prove the last part of the theorem, suppose that $\zeta = \{\zeta_{1},\dots, \zeta_{q}\}$ is another system of $\Delta$-$\sigma$-generators of $L/K$, that is,
$L = K\langle \eta_{1},\dots,\eta_{s}\rangle = K\langle \zeta_{1},\dots,\zeta_{q}\rangle$. Let $$\Phi_{\zeta}(t_{1},\dots, t_{p+1}) = \D\sum_{i_{1}=0}^{m_{1}}\dots  \D\sum_{i_{p}=0}^{m_{p}}\D\sum_{i_{p+1}=0}^{n}b_{i_{1}\dots i_{p+1}} {t_{1}+i_{1}\choose i_{1}}\dots {t_{p+1}+i_{p+1}\choose i_{p+1}}$$
be the dimension polynomial of our $\Delta$-$\sigma$-field  extension associated with the system of generators $\zeta$. Then there exist positive integers
$h_{1}, \dots, h_{p+1}$ such that $\eta_{i} \in K(\bigcup_{j=1}^{q}\Lambda(h_{1},\dots, h_{p+1})\zeta_{j})$ and $\zeta_{k} \in K(\bigcup_{j=1}^{s}\Lambda(h_{1},\dots,  h_{p+1})\eta_{j})$ for any $i=1,\dots, s$ and $k=1,\dots, q$, whence $\Phi_{\eta}(r_{1},\dots, r_{p+1}) \leq  \Phi_{\zeta}(r_{1}+h_{1},\dots, r_{p+1}+h_{p+1})$ and
$\Phi_{\zeta}(r_{1},\dots, r_{p+1}) \leq  \Phi_{\eta}(r_{1}+h_{1},\dots, r_{p+1}+h_{p+1})$ for all sufficiently large $(r_{1},\dots, r_{p+1})\in {\bf N}^{p+1}$. Now the statement of the third part of Theorem 3.1 follows from the fact that for any element $(k_{1},\dots, k_{p+1})\in E_{\eta}'$, the term ${t_{1}+k_{1}\choose k_{1}}\dots {t_{p+1}+k_{p+1}\choose k_{p+1}}$ appears in $\Phi_{\eta}(t_{1},\dots, t_{p+1})$ and $\Phi_{\zeta}(t_{1},\dots, t_{p+1})$ with the same coefficient $a_{k_{1}\dots k_{p+1}}$. The equality of the coefficients of the corresponding terms of total degree $d = deg\,\Phi_{\eta}$ in $\Phi_{\eta}$ and $\Phi_{\zeta}$ can be shown in the same way as in the proof of Theorem 3.3.21 of \cite{Levin4}.
\end{proof}

\begin{example}
{\em Let us find the $\Delta$-$\sigma$-dimension polynomial that expresses the strength of the difference-differential equation
\begin{equation}
\frac{\partial^{2} y(x_{1}, x_{2})}{\partial x_{1}^{2}} + \frac{\partial^{2} y(x_{1}, x_{2})}{\partial x_{2}^{2}} + y(x_{1} + h) + a(x)  = 0
\end{equation}
over some $\Delta$-$\sigma$-field of functions of two real variables $K$, where the basic set of derivations
$\Delta = \{\delta_{1} = \frac{\partial}{\partial x_{1}}, \delta_{2} = \frac{\partial}{\partial x_{2}}\}$ has the partition
$\Delta = \{\delta_{1}\}\bigcup \{\delta_{1}\}$ and $\sigma$ consists of one automorphisms $\alpha: f(x_{1}, x_{2})\mapsto f(x_{1}+h, x_{2})\}$
($h\in {\bf R}$).

In this case, the associated $\Delta$-$\sigma$-extension $K\langle\eta\rangle/K$ is $\Delta$-$\sigma$-isomorphic to the field of fractions of  $K\{y\}/[\alpha y + \delta_{1}^{2}y + \delta_{2}^{2}y + a]$ (the element $a\in K$ corresponds to the function $a(x)$).
Applying Proposition 5.10 we obtain that the characteristic set of the defining ideal of the corresponding $\Delta$-$\sigma$-extension
$K\langle\eta\rangle/K$ consists of the $\Delta$-$\sigma$-polynomials $g_{1} = \alpha y + \delta_{1}^{2}y + \delta_{2}^{2}y + a$ and
$g_{2} = \alpha^{-1}g_{1} = \alpha^{-1}\delta_{1}^{2}y + \alpha^{-1}\delta_{2}^{2}y + y + \alpha^{-1}(a)$.
With the notation of the proof of Theorem 3.1, the application of the procedure described in this proof, Theorem 4.6(iii), and formula (4.4)
leads to the following expressions for the numbers of elements of the sets $U_{r_{1}r_{2}r_{3}}^{(1)}$ and $U_{r_{1}r_{2}r_{3}}^{(2)}$:
$Card\,U_{r_{1}r_{2}r_{3}}^{(1)} = r_{1}r_{2}  + 2r_{2}r_{3} + r_{1} + r_{2} + 2r_{3} + 1$
and $Card\,U_{r_{1}r_{2}r_{3}}^{(2)} = 4r_{1}r_{3}  + 2r_{2}r_{3} - 2r_{3}$ for all sufficiently large $(r_{1}, r_{2}, r_{3})\in {\bf N}^{3}$.
Thus, the strength of equation (5.2) corresponding to the given partition of the basic set of derivations is expressed by the $\Delta$-$\sigma$-polynomial $$\Phi_{\eta}(t_{1}, t_{2}, t_{3}) = t_{1}t_{2} + 4t_{1}t_{3} + 4t_{2}t_{3} + t_{1} + t_{2} + 1.$$}
\end{example}

\begin{example} {\em Let $K$ be a difference-differential ($\Delta$-$\sigma$-) field where the basic set of derivations $\Delta =
\{\delta_{1}, \delta_{2}\}$ is considered together with its partition

\begin{equation}\Delta = \{\delta_{1}\}\bigcup \{\delta_{1}\}\end{equation}
and $\sigma = \{\alpha\}$ for some automorphism $\alpha$ of $K$. Let
$L = K\langle \eta \rangle$ be a $\Delta$-$\sigma$-field extension with the
defining equation
\begin{equation}\delta_{1}^{a}\delta_{2}^{b}\alpha^{c}\eta +
\delta_{1}^{a}\delta_{2}^{b}\alpha^{-c}\eta + \delta_{1}^{a}\delta_{2}^{b+c}\eta + \delta_{1}^{a+c}\delta_{2}^{b}\eta= 0
\end{equation}
where  $a$, $b$, and $c$ are positive integers. Let $\Phi_{\eta}(t_{1}, t_{2}, t_{3})$ denote the corresponding
difference-differential dimension polynomial (which expresses the strength of equation (5.4) with respect to the
given partition of the set of basic derivations $\Delta$). In order to compute $\Phi_{\eta}$, notice, first , that
the defining $\Delta$-$\sigma$-ideal $P$ of the extension $L/K$ is the linear $\Delta$-$\sigma$-ideal of $K\{y\}$
generated by the $\Delta$-$\sigma$-polynomial
$$f = \delta_{1}^{a}\delta_{2}^{b}\alpha^{c} y +
\delta_{1}^{a}\delta_{2}^{b}\alpha^{-c} y + \delta_{1}^{a}\delta_{2}^{b+c} y + \delta_{1}^{a+c}\delta_{2}^{b} y.$$
By Proposition 5.10, the characteristic set of the ideal $P$ consists of $f$ and
$$\alpha^{-1}f =  \alpha^{-(c+1)}\delta_{1}^{a}\delta_{2}^{b} y +
\delta_{1}^{a}\delta_{2}^{b}\alpha^{c-1} y + \delta_{1}^{a}\delta_{2}^{b+c}\alpha^{-1}y + \delta_{1}^{a+c}\delta_{2}^{b}\alpha^{-1}y.$$
The procedure described in the proof of Theorem 3.1 shows that $Card\,U_{r_{1}r_{2}r_{3}}^{(1)} = \phi_{A}(r_{1}, r_{2}, r_{3})$ for all
sufficiently large $(r_{1}, r_{2}, r_{3})\in {\bf N}^{3}$, where
$\phi_{A}(t_{1}, t_{2}, t_{3})$  is the dimension polynomial of the set
$A = \{(a, b, c),  (a, b, -(c+1)\,)\}\subseteq {\bf N}^{2}\times{\bf Z}$.
Applying Theorem 4.6(iii), and formula (4.4) we obtain that
$\phi_{A}(t_{1}, t_{2}, t_{3}) = 2ct_{1}t_{2} + 2bt_{1}t_{3} + 2at_{2}t_{3} + (b+2c-2bc)t_{1} +(a+2c-2ac)t_{2} + (2a+2b-2ab)t_{3} +
a+b+2c-ab-2ac-2bc+2abc$. The computation of $Card\,U_{r_{1}r_{2}r_{3}}^{(2)}$ with the use of the method of inclusion and exclusion described in the
proof of Theorem 3.1 yields the following:  $Card\,U_{r_{1}r_{2}r_{3}}^{(2)} = (2r_{3}-2c+1)[c(r_{2}-b+1) + c(r_{1}-a+1) - c^{2}]$ for all
sufficiently large $(r_{1}, r_{2}, r_{3})\in {\bf N}^{3}$. Therefore, the $\Delta$-$\sigma$-dimension polynomial of the extension $L/K$, which expresses the strength of equation (5.4), is as follows.
$$\Phi_{\eta}(t_{1}, t_{2}, t_{3}) = 2ct_{1}t_{2} + 2(b+c)t_{1}t_{3} + 2(a+c)t_{2}t_{3} + (b+3c-2bc-c^{2})t_{1}$$
$$+(2a+2b+4c -2ab - 2ac - 2bc - 2c^{2})t_{3} + a+b+4c-ab-3ac-3bc$$
\begin{equation}
+ (a+3c-2ac -2c^{2})t_{2} + +2abc +2ac^{2} + 2bc^{2} + 2c^{3} - 5c^{2}.\hspace{0.3in}
\end{equation}
The computation of the Kolchin-type univariate $\Delta$-$\sigma$-dimension polynomial (see Theorem 2.1) via the method of K\"ahler differentials described in
~\cite[Section 6.5]{KLMP} (by mimicking Example 6.5.6 of ~\cite{KLMP}) leads to the following result:
\begin{equation}
\phi_{\eta|K}(t) = {\D\frac{D}{2}}t^{2} - {\D\frac{D(D-2)}{2}}t + {\D\frac{D(D-1)(D-2)}{6}}
\end{equation}
where $D = a+b+c$. In this case the polynomial $\phi_{\eta|K}(t)$ carries just one invariant $a+b+c$ of the extension $L/K$ while
$\Phi_{\eta}(t_{1}, t_{2}, t_{3})$ determines three such invariants: $c,\, b+c$, and $a+c$ (see Theorem 3.1(iii)\,),
that is, $\Phi_{\eta}$ determines all three parameters $a, b, c$ of the defining equation while $\phi_{\eta}(t)$ gives just the sum of these parameters.

The extension $K\langle \zeta\rangle/K$ with a $\Delta$-$\sigma$-generator $\zeta$, the same basic set $\Delta\bigcup\sigma$ ($\Delta =
\{\delta_{1}, \delta_{2}\}$, $\sigma = \{\alpha\}$), the same partition of $\Delta$ and defining equation
\begin{equation}\delta_{1}^{a+b}\alpha^{c}\zeta + \delta_{2}^{a+b}\alpha^{-c}\zeta = 0
\end{equation}
has the same univariate difference-dimension polynomial (5.6). However, its $\Delta$-$\sigma$-dimension polynomial is not only different, but also has different invariants described in part (iii) of Theorem 3.1:
$$\Phi_{\zeta}(t_{1}, t_{2}, t_{3}) = 2ct_{1}t_{2} + 2(a+b)t_{1}t_{3} + 2(a+b)t_{2}t_{3} + At_{1} + Bt_{2} + Ct_{3} + E$$
where $A = B = (a+b)(1-2c) + 2c$,\, $C = 2[1 - (a+b-1)^{2}]$,\,  and \,$E = 1 + 2c(a+b-1)^{2}$.

Two systems of algebraic difference-differential ($\Delta$-$\sigma$-) equations with coefficients from a $\Delta$-$\sigma$-field $K$ are said to be {\em equivalent} if there is a $\Delta$-$\sigma$-isomorphism between the $\Delta$-$\sigma$-field extensions of $K$ with these defining equations, which is identity on $K$. Our example shows that using a partition of the basic set of derivations and the computation of the corresponding multivariate $\Delta$-$\sigma$-dimension polynomials, one can determine that two systems of $\Delta$-$\sigma$-equations (see systems (5.4) and (5.7)\,) are not equivalent, even though they have the same univariate difference-dimension polynomial.}
\end{example}
\bigskip

\noindent\Large {\bf Acknowledges}
\normalsize
\medskip

\noindent This research was supported by the NSF Grant CCF 1016608

\end{document}